\documentclass[12pt]{article}
\usepackage{a4wide}

\usepackage{amsmath,amssymb,amsthm}
\usepackage[mathscr]{eucal}
\usepackage{graphics}
\usepackage[hypertex]{hyperref}
\usepackage{epsfig}
\usepackage{psfrag}

\author{Ievgen~V.~Bondarenko, Rostyslav~V.~Kravchenko\thanks{The author was partially supported by NSF grant 0503688}}
\title{\textbf{On Lebesgue measure of integral\\ self-affine sets}}

\newcommand{\R}{\mathbb{R}}
\newcommand{\Z}{\mathbb{Z}}
\newcommand{\nucl}{\mathcal{N}}
\newcommand{\gr}{\Gamma}

\newtheorem{theorem}{Theorem}

\newtheorem{cor}[theorem]{Corollary}

\theoremstyle{definition}

\newtheorem{example}{Example}

\begin{document}
\maketitle

\begin{abstract}
Let $A$ be an expanding integer $n\times n$ matrix and $D$ be a finite subset of $\Z^n$. The
self-affine set $T=T(A,D)$ is the unique compact set satisfying the equality $A(T)=\cup_{d\in D}
(T+d)$. We present an effective algorithm to compute the Lebesgue measure of the self-affine set
$T$, the measure of the intersection $T\cap (T+u)$ for $u\in\Z^n$, and the measure of the
intersection of self-affine sets $T(A,D_1)\cap T(A,D_2)$ for different sets $D_1,D_2\subset\Z^n$.\\


\end{abstract}


Let $A$ be an expanding integer $n\times n$  matrix, where expanding means that every eigenvalue has modulus greater than
$1$, and let $D$ be a finite subset of $\Z^n$. There exists a unique nonempty compact set $T=T(A,D)\subset\mathbb{R}^n$,
called (integral) \textit{self-affine set}, satisfying $A(T)=\cup_{d\in D} (T+d).$ It can be given explicitly by
\[
T=\left\{ \sum_{k=1}^{\infty} A^{-k}d_k : d_k\in D\right\}.
\]
The self-affine set $T$ with $|D|=|\det A|$ and of positive Lebesgue measure is called a
\textit{self-affine tile}. Self-affine tiles were intensively studied for the last two decades in
the context of self-replicating tilings, radix systems, Haar-type wavelets, etc.

The question of how to find the Lebesgue measure $\lambda(T)$ of the self-affine set $T$ was considered by Lagarias and
Wang in \cite{lag_wang:int_self_affine_I}, where some partial cases were studied. In particular, it was shown that
self-affine tiles have integer Lebesgue measure. He, Lau and Rao \cite{he_rao:self_affine} reduced the problem of finding
$\lambda(T)$ to the case when $D$ is a coset transversal for $\Z^n/A(\Z^n)$. The last case was treated by Gabardo and Yu
\cite{gab_yu:natur_til} and in more general settings by Bondarenko and Kravchenko \cite{BK:meas}. The positivity of the
Lebesgue measure of self-affine sets was also studied in
\cite{lag_wang:self_affine,kirat:leb_meas,deng_he:int_self_affine}.


In this note, we present a simple method to compute the Lebesgue measure $\lambda(T)$ of the self-affine set~$T$. We
construct a finite labeled graph (automaton) and show that $\lambda(T)$ is equal to the uniform Bernoulli measure of the
left-infinite sequences which can be read along paths in this graph. Similar graphs when $D$ is a coset transversal were
constructed in \cite{gab_yu:natur_til,canon_numsys} and other papers. In addition this method allows to find the measure
of the intersection $T\cap (T+u)$ for $u\in\Z^n$, and the measure of the intersection of self-affine sets $T(A,D_1)\cap
T(A,D_2)$ for different sets $D_1,D_2\subset\Z^n$. Our construction seems to be very natural and actually works for any
contracting self-similar group action (here the self-affine sets correspond to the self-similar actions of $\Z^n$, see
\cite[Section 6.2]{self_sim_groups} and \cite{BK:meas}).

We proceed as follows. If the set $D$ does not contain all coset representatives of $\Z^n/
A(\Z^n)$, we extend it to the set $K\supset D$ which does, and choose a coset transversal $C\subset
K$.

Construct a directed labeled graph (automaton) $\gr=\gr(A,K)$  with the set of vertices $\Z^n$, and we put a directed
edge from $u$ to $v$ for $u,v\in\Z^n$ labeled by the pair $(x,y)$ for $x,y\in K$ if $u+x=y+Av$. The \textit{nucleus} of
the graph $\gr$ is the subgraph (subautomaton) $\nucl$ spanned by all cycles of $\gr$ and all vertices that can be
reached following directed paths from the cycles. Since the matrix $A$ is expanding the nucleus $\nucl$ is a finite graph
and it can be algorithmically computed. Indeed, if $u+x=y+Av$ then
\[
\|v\|<\|u\| \ \mbox{ whenever } \ \|u\|>(1-\|A^{-1}\|)^{-1}\max_{x,y\in{K}}\|A^{-1}(x-y)\|,
\]
and the nucleus $\nucl$ is contained in the ball centered at the origin of radius given by the
right-hand side above. Remove every edge in $\nucl$ whose label is not in $C\times D$, and replace
every label $(a,b)$ by $a$. We get some finite graph $\nucl_D$ whose edges are labeled by elements
of the set $C$.

Let $C^{-\omega}$ be the space of all left-infinite sequences $\ldots x_2x_1$, $x_i\in C$, with the
product topology of discrete sets. Let $\mu$ be the uniform Bernoulli measure on $C^{-\omega}$,
i.e. the product measure with $\mu(x)=1/|C|$ for every $x\in C$. For every vertex $v$ of the graph
$\nucl_D$ denote by $F_v$ the set of all left-infinite sequences which can be read along
left-infinite paths in $\nucl_D$ that end in $v$. The sets $F_v$ are closed in $C^{-\omega}$, thus
compact and measurable.

\begin{theorem}
The Lebesgue measure of the self-affine set $T$ is equal
\[
\lambda(T)=\sum_{v\in \nucl_D} \mu(F_v).
\]
\end{theorem}
\begin{proof}
Consider the map $\Phi:K^{-\omega}\times\Z^n\rightarrow\R^n$ given by the rule
\[
\Phi(\ldots x_2x_1, v)=v+A^{-1}x_1+A^{-2}x_2+\ldots,
\]
where $x_i\in K$ and $v\in\Z^n$. Since $\Z^n=K+A(\Z^n)$, the map $\Phi$ is onto (see \cite{lag_wang:self_affine} or
\cite[Section~6.2]{self_sim_groups}). Two elements $\xi=(\ldots x_2x_1, v)$ and $\zeta=(\ldots y_2y_1, u)$ for
$x_i,y_i\in K$ and $v,u\in\Z^n$ represent the same point $\Phi(\xi)=\Phi(\zeta)$ in $\R^n$ if and only if there is a
finite subset $B\subset\Z^n$ and a sequence $\{v_m\}_{m\geq 1}\in B$ such that there exists the path
\begin{equation}\label{eq_path}
v_m\xrightarrow{(x_m,y_m)} v_{m-1}\xrightarrow{(x_{m-1},y_{m-1})} \ldots \xrightarrow{(x_2,y_2)}
v_1 \xrightarrow{(x_1,y_1)} u-v
\end{equation}
in the graph $\gr$ for every $m\geq 1$. Indeed, this path implies that
\begin{equation}\label{eq_sum_m}
v_m+x_m+Ax_{m-1}+\ldots+A^{m-1}x_1+A^mv = y_m+Ay_{m-1}+\ldots+A^{m-1}y_1+A^mu.
\end{equation}
Applying $A^{-m}$ and using the facts that $A^{-1}$ is contracting and the sequence $\{v_m\}_{m\geq 1}$ attains a finite
number of values, we get the equality  $\Phi(\xi)=\Phi(\zeta)$. For the converse, we choose $v_m$ such that
(\ref{eq_sum_m}) holds, and using the equality $\Phi(\xi)=\Phi(\zeta)$ we get that $\{v_m\}_{m\geq 1}$ attains a finite
number of values. Notice that since the set $B$ is assumed to be finite, every element $v_m$ lies either on a cycle or
there is a directed path from a cycle to $v_m$. In particular, all elements $v_m$ should belong to the nucleus $\nucl$,
and we have that the elements $\xi$ and $\zeta$ represent the same point in $\R^n$ if and only if there exists a
left-infinite path in $\nucl$ labeled by $(\ldots x_2x_1, \ldots y_2y_1)$ and ending in $u-v$.

Take the restriction $\Phi_C:C^{-\omega}\times\Z^n\rightarrow\R^n$ of the map $\Phi$ on the subset
$C^{-\omega}\times\Z^n$. Since $\Z^n=C+A(\Z^n)$, the map $\Phi_C$ is also onto, and this gives an encoding of points in
$\R^n$ by elements of $C^{-\omega}\times\Z^n$. Consider the uniform Bernoulli measure $\mu$ on the space $C^{-\omega}$
and the counting measure on the group $\Z^n$, and put the product measure on the space $C^{-\omega}\times\Z^n$. Since the
set $C$ is a coset transversal, the push-forward of this measure under $\Phi_C$ is the Lebesgue measure on $\R^n$ (see
\cite[Proposition 25]{BK:meas}). Hence to find the Lebesgue measure of the self-affine set $T$ it is sufficient to find
the measure of its preimage in $C^{-\omega}\times\Z^n$. However, $T$ is equal to $\Phi(D^{-\omega}\times 0)$, and hence
the sequence $(\ldots x_2x_1, v)$ for $x_i\in C$ and $v\in\Z^n$ represents a point in $T$ if and only if there exists a
left-infinite path in the nucleus $\nucl$, which ends in $-v$ and is labeled by $(\ldots x_2x_1,\ldots y_2y_1)$ for some
$y_i\in D$. Hence
\begin{equation}\label{eq_preimage}
\Phi^{-1}_C(\Phi(D^{-\omega}\times 0))=\bigcup_{v\in \nucl_D} F_v\times \{-v\},
\end{equation}
and the statement follows.
\end{proof}

The Bernoulli measure of the sets $F_v$ for any finite graph $\gr=(V,E)$ can be effectively computed (see
\cite[Section~2]{BK:meas}). First, we can assume that the graph is left-resolving, i.e. for every vertex $v\in V$ the
incoming edges to $v$ have different labels. Indeed, for any finite graph $\gr=(V,E)$ there exists a left-resolving graph
$\gr'=(V',E')$ with the property that for every $v\in V$ there exists $v'\in V'$ such that $F_v=F_{v'}$, and this graph
can be easily constructed (here every vertex $v'$ corresponds to some subset of $V$, see
\cite[Section~2.3]{autom_theory}). For a left-resolving graph the vector $(\mu(F_v))_{v\in V}$ (if it is nonzero) is the
left eigenvector of the adjacency matrix of the graph for the eigenvalue $|C|=|\det A|$. This eigenvector is uniquely
defined if we know its entries $\mu(F_v)$ for vertices $v$ in the strongly connected components without incoming edges.
For every such a component $\gr'$, we have $F_v=C^{-\omega}$ and $\mu(F_v)=1$ for every vertex $v$ in $\gr'$ if inside
this component every vertex has incoming edges labeled by every element of the set $C$, and $\mu(F_v)=0$ otherwise. In
particular, the entries $\mu(F_v)$ are rational numbers, and we recover the following result of
\cite{he_rao:self_affine}.


\begin{cor}
Every self-affine set has rational Lebesgue measure.
\end{cor}

It is also easy to check when the measure of $T$ is non-zero without calculating its precise value but just looking at
the left-resolving graph (not the graph $\nucl_D$) constructed above. The measure $\lambda(T)$ will be positive if and
only if there exists a strongly connected component such that inside this component every vertex has incoming edges
labeled by every letter of the alphabet.

\begin{example}
Let $A=(3)$ and $D=\{0,1,5,6\}$. The self-affine set $T$ is $[0,\frac{4}{3}]\cup [\frac{5}{3},3]$,
and $\lambda(T)=8/3$. Choose $K=D$ and the coset transversal $C=\{0,1,5\}$. The associated
automaton $\nucl_D$ is shown in Figure~\ref{fig_Example1}. Here $\mu(F_0)=1$, $\mu(F_1)=1/3$,
$\mu(F_2)=1/8$, $\mu(F_{-1})=7/12$, $\mu(F_{-2})=5/8$, and $\mu(F_{-3})=\mu(F_3)=0$.
\end{example}

\begin{figure}
\begin{center}
\psfrag{0}{\small{$0$}} \psfrag{1}{\small{$1$}} \psfrag{3}{\small{$3$}} \psfrag{5}{\small{$5$}}
\psfrag{0,1,5}{\small{$0,1,5$}} \psfrag{0,1}{\small{$0,1$}} \psfrag{0,5}{\small{$0,5$}}
\psfrag{2}{\small{$2$}} \psfrag{-1}{\small{$-1$}} \psfrag{-2}{\small{$-2$}}
\psfrag{-3}{\small{$-3$}} \epsfig{file=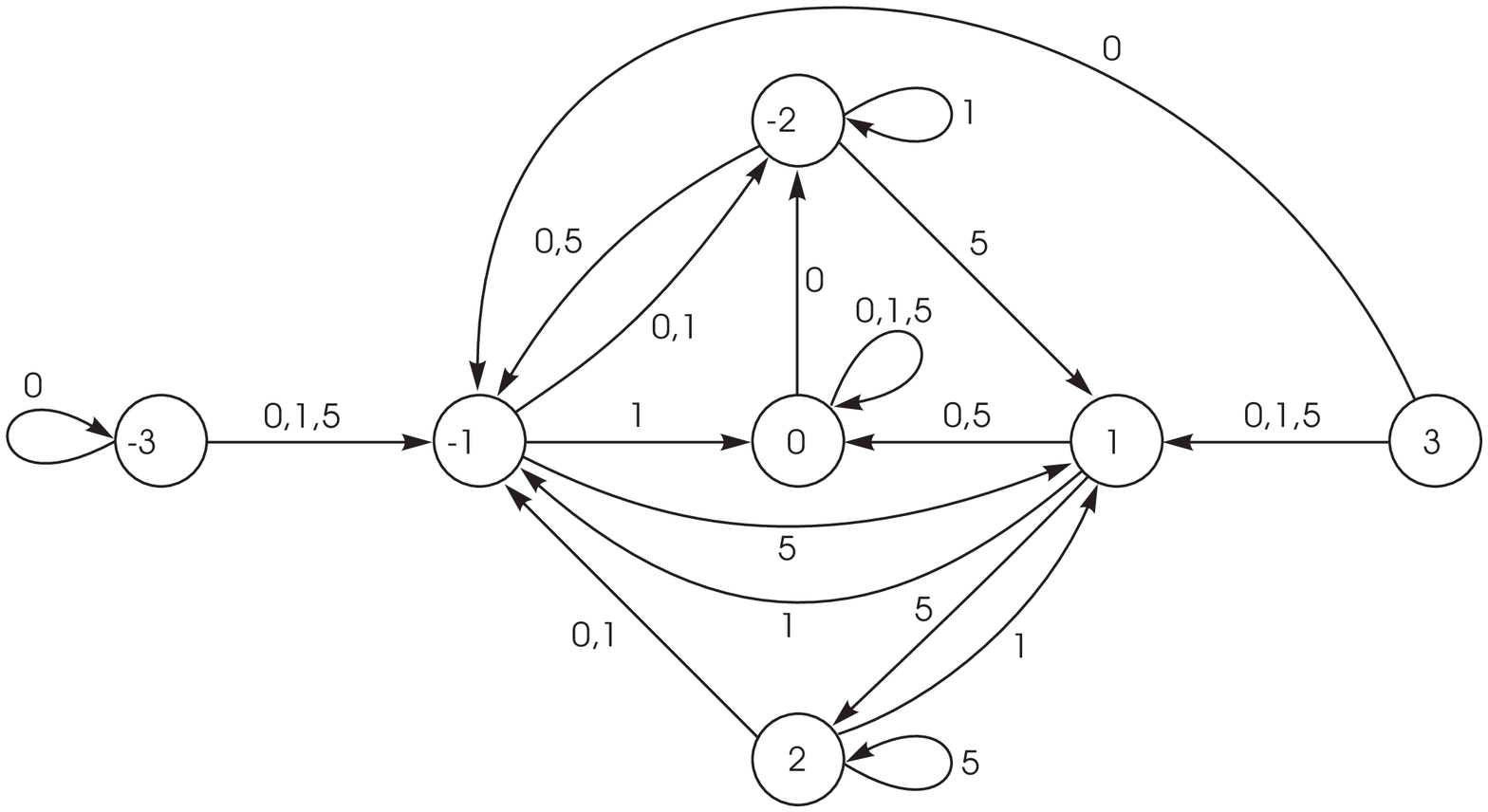,height=200pt}\caption{The graph $\nucl_D$
for $A=(3)$ and $D=\{0,1,5,6\}$} \label{fig_Example1}
\end{center}
\end{figure}

The above method can be used to find $\lambda(T\cap (T+u))$ for $u\in\Z^n$. The set $T+u$ is the
image of the set $D^{-\omega}\times u$, and its preimage under $\Phi_C$ can be described as in
(\ref{eq_preimage}). In particular
\[
\lambda(T\cap(T+u))=\sum_{{\scriptsize\begin{array}{c}
                              v_1,v_2\in\nucl_D \\
                              u=v_2-v_1
                            \end{array}}} \mu (F_{v_1}\cap F_{v_2}).
\]
Similarly, one can find the measure of the intersection of self-affine sets $T_1=T(A,D_1)$ and
$T_2=T(A,D_2)$ for different sets $D_1,D_2\subset\Z^n$. We take a set $E$ which contains $D_1$,
$D_2$, and some coset transversal $C$, and as above we construct the nucleus $\nucl$ and its
subgraphs $\nucl_{D_1}$ and $\nucl_{D_2}$. Then
\[
\lambda(T_1\cap T_2)=\sum_{v\in\nucl} \mu(F^{(1)}_v\cap F^{(2)}_v),
\]
where $F^{(i)}_v$ is calculated in the graph $\nucl_{D_i}$. Hence these two problems are reduced to the question of how
to find the measure of the intersection $F^{(1)}_{v_1}\cap F^{(2)}_{v_2}$, where each set $F^{(i)}_{v_i}$ is defined in
some finite graph $\gr^{(i)}=(V^{(i)}, E^{(i)})$ with its vertex $v_i$. One can construct a new finite graph~$\gr$
(sometimes called the labeled product of graphs $\gr^{(i)}$) with the set of vertices $V^{(1)}\times V^{(2)}$, where we
put an edge $(u_1,u_2)\stackrel{x}\rightarrow (w_1,w_2)$ for every edges $u_1\stackrel{x}\rightarrow w_1$ in $\gr^{(1)}$
and $u_2\stackrel{x}\rightarrow w_2$ in $\gr^{(2)}$. Then $F_{(v_1,v_2)}=F^{(1)}_{v_1}\cap F^{(2)}_{v_2}$ (see
\cite[Section~3.2]{autom_theory}).

\bibliographystyle{plain}

\end{document}